\documentclass[11pt,reqno]{amsart} %
\usepackage{amsmath,amscd}
\usepackage{amsthm}
\usepackage{amssymb}
\usepackage{amsfonts}
\usepackage{latexsym}
\usepackage{url}
\usepackage{latexsym}
\usepackage{graphicx,color,subfig}
\usepackage{enumerate,fancybox}
\usepackage[colorlinks=true, citecolor=blue]{hyperref} 

\DeclareMathOperator\dist{dist}

\newtheorem{theorem}{Theorem}[section]

\newtheorem{lemma}{Lemma}[section]

\newtheorem{example}{Example}[section]
\newtheorem{remark}{Remark}[section]

\numberwithin{equation}{section}

\begin{document}
\title[]
{Asymptotics for Hessenberg matrices for the Bergman shift operator on Jordan regions}

\date{\today}

\author[E.\ Saff]{Edward B.\ Saff}
\address{Center for Constructive Approximation\\
         Department of Mathematics\\
         Vanderbilt University\\
         1326 Stevenson Center\\
         37240 Nashville\\
         USA}
\email{edward.b.saff@Vanderbilt.Edu}
\urladdr{http://www.math.vanderbilt.edu/\textasciitilde esaff/}
\author[N. Stylianopoulos]{Nikos Stylianopoulos}
\address{Department of Mathematics and Statistics,
         University of Cyprus, P.O. Box 20537, 1678 Nicosia, Cyprus}
\email{nikos@ucy.ac.cy}
\urladdr{http://ucy.ac.cy/~nikos}

\thanks{The research of the first author was supported, in part, by U.S.\ National Science Foundation
grants DMS-0808093 and DMS-1109266}

\begin{abstract}
Let $G$ be a bounded Jordan domain in the complex plane. The Bergman polynomials $\{p_n\}_{n=0}^\infty$ of $G$
are the orthonormal polynomials with respect to the area measure over $G$. They are
uniquely defined by the entries of an infinite upper Hessenberg matrix $M$. This matrix represents the
 Bergman shift operator of $G$. The main purpose of the paper is to describe and analyze a close relation
between $M$ and the Toeplitz matrix with symbol the normalized conformal map of the exterior of the unit
circle onto the complement of $\overline{G}$. Our results are based on the strong asymptotics of $p_n$.
As an application, we describe and analyze an algorithm for recovering the shape of $G$ from its area moments.
\end{abstract}

\maketitle
\allowdisplaybreaks
\section{Introduction}\label{section:intro}

Let $G$ be a bounded simply-connected domain in the complex plane $\mathbb{C}$,
whose boundary $\Gamma:=\partial G$ is a Jordan curve and let
$\{p_n\}_{n=0}^{\infty}$ denote the sequence of  Bergman polynomials of
$G$. This is the unique sequence of polynomials
\begin{equation}\label{eq:pndef}
p_n(z) = \lambda_n z^n+ \cdots, \quad \lambda_n>0,\quad n=0,1,2,\ldots,
\end{equation}
that are orthonormal with respect to the inner product
$$
\langle f,g\rangle := \int_G f(z) \overline{g(z)} dA(z),
$$
where $dA$ stands for the area  measure.  We denote by $L_a^2(G)$ the Hilbert space of all functions $f$
analytic in $G$ for which
$$
\|f\|_{L^2(G)}:=\langle f,f\rangle^{1/2}<\infty,
$$
and recall (cf.\ \cite{Gabook87}) that the polynomials $\{p_n\}_{n=0}^\infty$ form a complete orthonormal
system for $L_a^2(G)$.

Let $\Omega:=\overline{\mathbb{C}}\setminus\overline{G}$ denote the complement of $\overline{G}$ in $\overline{\mathbb{C}}$ and let $\Phi$
denote the conformal map
$\Omega\to\Delta:=\{w:|w|>1\}$, normalized so that near infinity
\begin{equation}\label{eq:Phi}
\Phi(z)=\gamma z+\gamma_0+\frac{\gamma_1}{z}+\frac{\gamma_2}{z^2}+\cdots,\quad \gamma>0.
\end{equation}
Finally, let $\Psi:=\Phi^{-1}:\Delta\to\Omega$ denote the inverse conformal map. Then,
\begin{equation}\label{eq:Psi}
\Psi(w)=bw+b_0+\frac{b_1}{w}+\frac{b_2}{w^2}+\cdots, \quad |w|> 1,
\end{equation}
with
\begin{equation}\label{eq:b-def}
b=1/\gamma=\textup{cap}(\Gamma),
\end{equation}
where $\textup{cap}(\Gamma)$ denotes the \textit{(logarithmic) capacity} of $\Gamma$.

On $L_a^2(G)$ we consider the multiplication by $z$ operator (also known as the \textit{Bergman shift operator})
$\mathcal{M}:f\to zf$. Note that $\mathcal{M}$ defines a bounded, noncompact,  linear operator on $L_a^2(G)$
and that
\begin{equation}\label{eq:essspM}
\sigma_{ess}(\mathcal{M})=\Gamma;
\end{equation}
see \cite{AxCoMcDo}, where we use $\sigma_{ess}(L)$ to denote the  \textit{essential spectrum} of a bounded
linear operator $L;$ that is, the set of all $\lambda\in\mathbb{C}$ for which $L-\lambda I$ is not a Fredholm operator. For
the operators we consider, the essential spectrum is the same as the continuous spectrum.

We also consider the matrix representation of $\mathcal{M}$ in terms of the orthonormal basis
$\{p_n\}_{n=0}^\infty$. This induces the upper Hessenberg matrix

\begin{equation}\label{eq:Mmat}
{M}=\left[%
\begin{array}{cccccc}
   b_{00}& b_{01}& b_{02}& b_{03}& b_{04}&\cdots\\
   b_{10}& b_{11}& b_{12}& b_{13}& b_{14}&\cdots \\
    0    & b_{21}& b_{22}& b_{23}& b_{24}&\cdots \\
    0    &  0    & b_{32}& b_{33}& b_{34}&\cdots \\
    0    &  0    &   0   & b_{43}& b_{44}&\cdots \\
   \vdots& \vdots& \vdots& \ddots& \ddots&\ddots
\end{array}%
\right],
\end{equation}
where
\begin{equation}\label{eq:bij}
b_{k,j}=\langle zp_j,p_k\rangle,\quad k\ge 0,\ j\ge 0.
\end{equation}
Note that $b_{k,j}=0$ for $k\ge j+2$ and that
\begin{equation}\label{eq:zp_n}
zp_n(z)=\sum_{k=0}^{n+1}b_{k,n}p_k(z).
\end{equation}
In particular,
\begin{equation}\label{eq:bn+1n}
b_{n+1,n}=\frac{\lambda_n}{\lambda_{n+1}}>0,\quad n=0,1,\ldots.
\end{equation}
It follows
\begin{equation}\label{eq:zpn=}
 b_{n+1,n}p_{n+1}(z)=zp_n(z)-\sum_{k=0}^nb_{k,n}p_k(z)
\end{equation}
and, hence, the entries of $M$ define uniquely the sequence of Bergman polynomials of $G$,
in the sense that $p_{n+1}$, $n=0,1,\ldots$, can be computed recursively from (\ref{eq:zpn=}).

It is shown in \cite{PuSt} and \cite{KhSt} (see also \cite[Thm 2.4]{St-CR10}) that except for some trivial cases,
 the matrix (\ref{eq:Mmat})
is not banded; i.e.,  the $p_n$'s do not satisfy a recurrence relation of bounded length. It is also well-known
that the eigenvalues
of the $n\times n$ principal submatrix of $M$ coincide with the zeros of $p_n(z)$.

Our goal is to investigate the asymptotic behavior of the entries in the matrix $M$. In particular, we show that
if the boundary of $G$ is piecewise analytic without cusps, then all the diagonals (sub, super and main) have limits which are
the coefficients of the Laurent expansion (\ref{eq:Psi}) of the inverse conformal map $\Psi$:
\begin{equation}\label{eq:limit-subdia}
\lim_{n\to\infty}b_{n+1,n}=b \quad\text{and}\quad\lim_{n\to\infty}b_{n-k,n}=b_k,\quad k=0,1,\ldots.
\end{equation}

A potential application of (\ref{eq:limit-subdia}) is in the area of geometric tomography, where the following
inverse problem arises: Given a finite number of complex moments
\begin{equation}\label{eq:area-moments}
\mu_{kj}:=\langle z^k,z^j\rangle=\int_G
z^k\overline{z}^j\,dA(z),\quad k,j=0,1,\ldots,
\end{equation}
how can one approximate the region $G$ that generated these moments?
Regarding existence and uniqueness, we note a result of Davis and Pollak \cite{DaPo} stating that the infinite matrix
$[\mu_{m,k}]_{m,k=0}^\infty$ defines uniquely the curve $\Gamma$.
By utilizing the given moments to compute Bergman polynomials, and thereby a principal submatrix of $M$,
the subdiagonals of the submatrix will provide an approximation to the Laurent coefficients of the mapping
of the unit circumference onto the boundary of $G$. We will discuss this procedure
in Section~\ref{sec:recovery}.

We note that there is a one-to-one correspondence between the complex moments (\ref{eq:area-moments}) and the real moments
\begin{equation}\label{eq:real-moments}
\tau_{mn}:=\int_G
x^my^n\,dxdy,\quad m,n=0,1,\ldots.
\end{equation}
Namely,
\begin{equation}\label{eq:tau-to-mu}
\mu_{m,n}=\sum_{j=0}^m\sum_{k=0}^n i^{m-j}i^{n-k}\binom{m}{j}\binom{n}{k}\tau_{j+k,m+n-j-k},\quad i:=\sqrt{-1},
\end{equation}
or, in the inverse direction,
\begin{equation}\label{eq:mu-to-tau}
\tau_{m,n}=(-i)^n2^{-m-n}\sum_{j=0}^m\sum_{k=0}^n \binom{m}{j}\binom{n}{k}\mu_{j+k,m+n-j-k};
\end{equation}
see \cite{DaPo}. Thus, the moments in (\ref{eq:area-moments}) will uniquely determine the moments in (\ref{eq:real-moments})
and vice-versa.


The Faber polynomials $\{F_n\}_{n=0}^\infty$ of $G$ are defined as the polynomial part of the expansion of
$\Phi^n(z)$, $n=0,1,\ldots$, near infinity, that is,
\begin{equation}\label{eq:PhinFn}
\Phi^n(z)=F_n(z)-E_n(z),\quad z\in\Omega,
\end{equation}
where
\begin{equation}\label{eq:FnEndef}
F_n(z)=\gamma^{n}z^n+\cdots\quad \text{and}\quad E_n(z)=O\left(\frac{1}{z}\right),\quad z\to\infty.
\end{equation}
The Faber polynomial of the 2nd kind, $G_n(z)$, is defined as the polynomial part of $\Phi^n(z)\Phi^\prime(z)$,
that is,
\begin{equation}\label{eq:PhinPhipGn}
G_n(z)=\Phi^n(z)\Phi^\prime(z)-H_n(z),\quad z\in\Omega,
\end{equation}
where
\begin{equation}\label{eq:Gndef}
G_n(z)=\gamma^{n+1}z^n+\cdots\quad\text{and}\quad H_n(z)=O\left(\frac{1}{z^2}\right),\quad z\to\infty.
\end{equation}
It follows immediately from (\ref{eq:PhinFn}) and (\ref{eq:PhinPhipGn}) that
\begin{equation}\label{eq:GnFn+1}
G_n(z)=\frac{F_{n+1}^\prime(z)}{n+1}\quad\textup{and}\quad H_n(z)=\frac{E_{n+1}^\prime(z)}{n+1}.
\end{equation}
It is well-known that the Faber polynomials of the 2nd kind satisfy the following recurrence relation
(see \cite[p.\ 52]{EVetna}):
\begin{equation}\label{eq:Gn-rr}
zG_n(z)=bG_{n+1}(z)+\sum_{j=0}^nb_jG_{n-j}(z),\quad G_0(z)\equiv b.
\end{equation}

Consider now the Toeplitz (and upper Hessenberg) matrix $T_\Psi$ defined by the continuous function $\Psi(w)$ on
$\mathbb{T}:=\{w:|w|=1\}$, that is,
\begin{equation}\label{eq:Tmat}
{T_\Psi}:=\left[%
\begin{array}{cccccc}
   b_{0}& b_{1}& b_{2}& b_{3}& b_{4}&\cdots\\
   b& b_{0}& b_{1}& b_{2}& b_{3}&\cdots \\
    0    & b& b_{0}& b_{1}& b_{2}&\cdots \\
    0    &  0    & b& b_{0}& b_{1}&\cdots \\
    0    &  0    &   0   & b& b_{0}&\cdots \\
   \vdots& \vdots& \vdots& \ddots& \ddots&\ddots
\end{array}%
\right].
\end{equation}
It follows from (\ref{eq:Gn-rr}) that the eigenvalues of the $n\times n$ principal submatrix of $T_\Psi$
coincide with the zeros of $G_n(z)$; see also \cite{Ul72}. This is a relation similar to the one connecting
the upper Hessenberg matrix $M$ with the Bergman polynomials $\{p_n\}_{n=0}^\infty$.

In \cite[\S7.8]{St-arXiv1202} it is shown that if $\Gamma$ is piecewise analytic without cusps, then
\begin{equation}\label{eq:bn-NS-Estimate}
|b_n|\le c_1(\Gamma)\frac{1}{n^{1+\omega}},\quad n\in\mathbb{N},
\end{equation}
where $\omega\pi$ ($0<\omega<2$) is the smallest exterior angle of $\Gamma$. (Hereafter, we use $c_k(\Gamma)$, $k=1,2,\ldots$, to denote a non-negative constant that depends only on $\Gamma$.)
Therefore, in this case, the symbol
$\Psi$ of the Toeplitz matrix $T_\Psi$ belongs to the Wiener algebra, which leads to the conclusion that
$T_\Psi$ defines a bounded linear operator on the Hilbert space $l^2$ and that
\begin{equation}\label{eq:essspT}
\sigma_{ess}(T_\Psi)=\Gamma;
\end{equation}
see e.g.\ \cite[p.~1--10]{Bo-Gr05}.

We end this section by noting a result, regarding a property of $H_n$, that we are going to use in
Section~\ref{sec:proofs}. A proof can be found in \cite[Lem. 2.1]{St-arXiv1202}.

\begin{lemma}\label{lem:Hn}
For any $n\in\mathbb{N}$, $H_n$ is analytic and square integrable in $\Omega$.
\end{lemma}

\section{Main results}\label{sec:Main}
In this section we state and discuss our main results. Their proofs are given in Section~\ref{sec:proofs}.
Section~\ref{sec:recovery} contains applications of our results to the recovery of planar regions.

From (\ref{eq:essspM}) and (\ref{eq:essspT}) it follows that
\begin{equation}\label{eq:esssp=}
\sigma_{ess}(\mathcal{M})=\sigma_{ess}(T_\Psi).
\end{equation}
The next theorem shows that the connection between the matrices $M$ and $T_\Psi$ is much more substantial.
\begin{theorem}\label{th:main-pw-analytic}
Assume that $\Gamma$ is piecewise analytic without cusps.
Then,  it holds as $n\to\infty$,
\begin{equation}\label{th:main-pw-analytic-1}
\sqrt{\frac{n+2}{n+1}}b_{n+1,n}=b+O\left(\frac{1}{{n}}\right),
\end{equation}
and for $k\ge 0$,
\begin{equation}\label{th:main-pw-analytic-2}
\sqrt{\frac{n-k+1}{n+1}}b_{n-k,n}=b_k+O\left(\frac{1}{\sqrt{n}}\right),
\end{equation}
where $O$ depends on $k$ but not on $n$. (See \text{(\ref{eq:0const})} for more precise estimates.)
\end{theorem}


Improvements in the order of convergence occur in cases when $\Gamma$ is smooth.
In order to state the corresponding results we need to introduce the smoothness class $C(q,\alpha)$ of Jordan curves.
We say that  $\Gamma$ belongs to  $C(q,\alpha)$,
$q\in\mathbb{N}$, if $\Gamma$ is defined by $z=g(s)$, where $s$ denotes arclength, with
$g^{(q)}\in \textup{Lip}\,\alpha$,
for some  $0<\alpha<1$. Then both $\Phi$ and $\Psi:=\Phi^{-1}$ are $q$ times continuously differentiable in
$\overline{\Omega}\setminus\{\infty\}$ and $\overline{\Delta}\setminus\{\infty\}$ respectively, with $\Phi^{(q)}$
and $\Psi^{(q)}$ in $\textup{Lip}\,\alpha$: see, e.g., \cite[p.\ 5]{Su74}.

\begin{theorem}\label{th:main-smooth}
Assume that $\Gamma\in C(p+1,\alpha)$, with $p+\alpha>1/2$.
Then,  it holds as $n\to\infty$,
\begin{equation}\label{th:main-sm-1}
\sqrt{\frac{n+2}{n+1}}b_{n+1,n}=b+O\left(\frac{1}{n^{2(p+\alpha)}}\right),
\end{equation}
and for $k\ge 0$,
\begin{equation}\label{th:main-sm-2}
\sqrt{\frac{n-k+1}{n+1}}b_{n-k,n}=b_k+O\left(\frac{1}{n^{p+\alpha}}\right),
\end{equation}
where $O$ depends on $k$ but not on $n$.
(See \text{(\ref{eq:0const-sm})} for more precise estimates.)
\end{theorem}

For the case of an analytic boundary $\Gamma$ further improved asymptotic results can be obtained. To state
these results we need to introduce some notation. For an analytic curve $\Gamma$ the mapping
$\Psi$ can be analytically continued as a conformal map to the exterior of some disk $\{w:|w|<\varrho\}$,
where $0<\varrho<1$. We denote by $L_\sigma$ the image of the circle $\{w:|w|=\sigma\}$ under the map $\Psi$.
In other words,
$$
L_\sigma:=\{z\in\mathbb{C}:|\Phi(z)|=\sigma\}.
$$

\begin{theorem}\label{th:main-analytic}
\footnote{This theorem, along with a sketch of its proof given in Section~\ref{sec:proof-analytic}, was
presented by the first author at the Joint Meeting of the AMS and MAA in Phoenix, January 2004.}
Assume that the boundary $\Gamma$ is analytic and let $\varrho<1 $ be the smallest index for which $\Phi$
is conformal in the exterior of $L_\varrho$. Then,  it holds as $n\to\infty$,
\begin{equation}\label{eq:main-analytic1}
\sqrt{\frac{n+2}{n+1}}\,b_{n+1,n}=b+O(\varrho^{2n}),
\end{equation}
and for $k\ge 0$,
\begin{equation}
\sqrt{\frac{n-k+1}{n+1}}\,b_{n-k,n}=b_k+O(\sqrt{n\log n}\varrho^n),
\end{equation}
where $O$ depends on $k$ but not on $n$.
(See \text{(\ref{eq:kn-analytic})} for more precise estimates.)
\end{theorem}

In the converse direction we have:
\begin{theorem}\label{th:inverse-analytic}
Assume that $\Gamma$ is a Jordan curve without zero interior angles. If
\begin{equation}\label{eq-inverse-analytic}
\limsup_{n\to\infty}\left|\sqrt{\frac{n+2}{n+1}}b_{n+1,n}-b\right|^{1/n}<1,
\end{equation}
then $\Gamma$ is analytic.
\end{theorem}

The following example shows that the inverse statement does not make sense for the main diagonal
of $M$.
\begin{example}\label{ex:RotSym}
Consider the case where the domain $G$ has $m$-fold rotational symmetry about the origin, for some $m\ge 2$.
\end{example}
This means that $\textup{e}^{i2\pi/m}z\in\Omega$, whenever $z\in\Omega$. Then, it is easy to see that
\begin{equation}\label{eq:rotsym-b0}
b_0=0\quad\text{and}\quad b_{n,n}=0,\quad n\ge m.
\end{equation}
Indeed, by using symmetry arguments it follows
\begin{equation}\label{eq:rotsym-Psi}
\Psi(\textup{e}^{i2\pi/m}w)=\textup{e}^{i2\pi/m}\Psi(w), \quad w\in\Omega,
\end{equation}
and for $n=km+j$, with $j=0,1,\ldots,m-1$,
\begin{equation}\label{eq:rotsym-pn}
p_n(z)=z^jq_k(z^m),\quad \deg(q_k)=k.
\end{equation}
The first relation in (\ref{eq:rotsym-b0}) follows at once from (\ref{eq:rotsym-Psi}).
For the second relation in (\ref{eq:rotsym-b0}),
observe that (\ref{eq:rotsym-pn}) implies for $n\ge m$ that
$$
p_n(z)=\lambda_nz^n+O(z^{n-m}),
$$
which, in turn,  yields $\langle z^{n+1},p_n\rangle=0$ and therefore  $b_{n,n}=\langle zp_n,p_n\rangle=0$.

\section{A recovery algorithm}\label{sec:recovery}
\noindent {\tt Reconstruction Algorithm}
\begin{enumerate}[1.]
\itemsep=5pt
\item
\textit{Start with a finite set of complex moments $\mu_{kj}$, $k,j=0,1,\ldots,n$; see (\ref{eq:area-moments}),
or, equivalently from a finite set of real moments $\tau_{kj}$, $k,j=0,1,\ldots,n$; see (\ref{eq:real-moments}).}
\item
\textit{Use the Arnoldi version of the Gram-Schmidt (GS) process, in the way indicated in \cite[\S7.4]{St-arXiv1202},
to construct the Bergman polynomials
$\{p_k\}_{k=0}^{n}$ from the moments $\mu_{kj}$, $k,j=0,1,\ldots,n$.
This involves at the $k$-step the orthonormalization of the set
$\{p_0,p_1,\ldots,p_{k-1},zp_{k-1}\}$, rather than the set of monomials $\{1,z,\ldots,z^{k-1},z^k\}$, as in the
conventional GS. This process, in particular, yields the inner products}
\begin{equation*}
b_{k,j}=\langle zp_j,p_k\rangle,\quad j= 0,1,\ldots,n,\ k= 0,\ldots,j+1.
\end{equation*}
\item
\textit{Choose a number $m$, $1<m<n$, and set}
\begin{equation}\label{eq:b-approx}
b^{(n)}:=\sqrt{\frac{n+2}{n+1}}b_{n+1,n}, \quad
b^{(n)}_k:=\sqrt{\frac{n-k+1}{n+1}}b_{n-k,n},\quad k=0,1,\ldots,m.
\end{equation}
\textit{(See Theorem~\ref{thm:recalg} and Remark~\ref{rem:recalg} below, for a suitable choice of $m$.)}
\item
\textit{Form}
\begin{equation}\label{eq:psi-approx}
\Psi^{(n)}_m(w):=b^{(n)}w+b^{(n)}_0+\frac{b^{(n)}_1}{w}+\ldots+\frac{b^{(n)}_m}{w^m}.
\end{equation}
\item
\textit{Approximate $\Gamma$ by $\Gamma^{(n)}_m$, where}
\begin{equation}\label{eq:Gamma-approx}
\Gamma^{(n)}_m:=\Psi^{(n)}_m(w),\quad w\in\mathbb{T}.
\end{equation}\end{enumerate}

\begin{remark}\label{rem:arnoldi}
We refer to \cite[\S7.4]{St-arXiv1202} for a discussion regarding the stability properties of the Arnoldi GS.
In particular, we note that the Arnoldi  GS does not suffer from the severe
ill-conditioning associated with the conventional GS as reported, for instance, by theoretical and numerical
evidence in \cite{PW86}.
\end{remark}

The following result justifies the use of the algorithm for analytic curves.

\begin{theorem}\label{thm:recalg}
Assume that $\Gamma$ is analytic, and let $\varrho<1$ be the smallest index for which $\Phi$
is conformal in the exterior of $L_\varrho$. Set $n=2m$. Then, for any $|w|\ge 1$ it holds that
\begin{equation}
|\Psi(w)-\Psi^{(n)}_m(w)|\le  c_1(\Gamma) \sqrt{m\log m}\,\varrho^m+c_2(\Gamma)|w|\varrho^{4m},
\end{equation}
where the constants $c_1(\Gamma)$ and $c_2(\Gamma)$ depend on $\Gamma$ only.
\end{theorem}

\begin{remark}\label{rem:recalg}
Similar estimates, as in the above theorem, can be obtained for the case where $\Gamma$ is piecewise analytic without cusps.
However, these estimates are too pessimistic compared with actual numerical evidence; see Figure~\ref{fig:square} below.
We were only able to rigorously show
that for an uniform error of order $O(1/\sqrt{m})$ we require the computation of the orthonormal polynomials up to degree
$m^{4+\omega}$, where $\omega\pi$ is the smallest exterior angle of $\Gamma$.
\end{remark}

For applications to the 2D image reconstruction arising from tomographic data we refer to \cite{GPSS}.
Here we highlight the performance of the reconstruction algorithm by applying it to the recovery of three curves,
coming from different classes: an analytic curve, one curve with corners and one curve with cusps.
For providing matter for comparison with the reconstruction algorithm of \cite[\S7.7]{St-arXiv1202} we have
chosen to present results for exactly the same curves as in \cite{St-arXiv1202}.
We note that the reconstruction algorithm of \cite{St-arXiv1202} is based on approximating first the exterior conformal
mapping $w=\Phi(z)$ in terms of the ratio $p_{n+1}(z)/p_n(z)$, cf. the estimates
(\ref{eqinthm:finepn})--(\ref{eqinthm:finepnii1}) below, and then on inverting the so-formed Laurent series in order to compute
an approximation of the inverse map $z=\Psi(w)$.

In each case we start by computing a finite set of complex moments (\ref{eq:area-moments}) up to degree $n$,
and then follow the steps 2--5 of the algorithm, taking $m=n/2$.
In all three examples the complex moments are known explicitly.
All computations  were carried out on a desktop PC, using the computing environment MAPLE.

In Figures~\ref{fig:ellipse}--\ref{fig:cupsed-hypo} we depict the computed approximation $\Gamma^{(n)}_m$
against the original curve $\Gamma$. The presented plots indicate that the above reconstruction algorithm constitutes
a valid method for recovering a shape from its partial moments. Even in the cusped case, pictured in Figure~\ref{fig:cupsed-hypo}, the fitting is
remarkably close, despite the low degree of the moment matrix used.

\begin{figure}[t]
\begin{center}
\begin{minipage}{2.5cm}
{\includegraphics[width=2.20\linewidth]{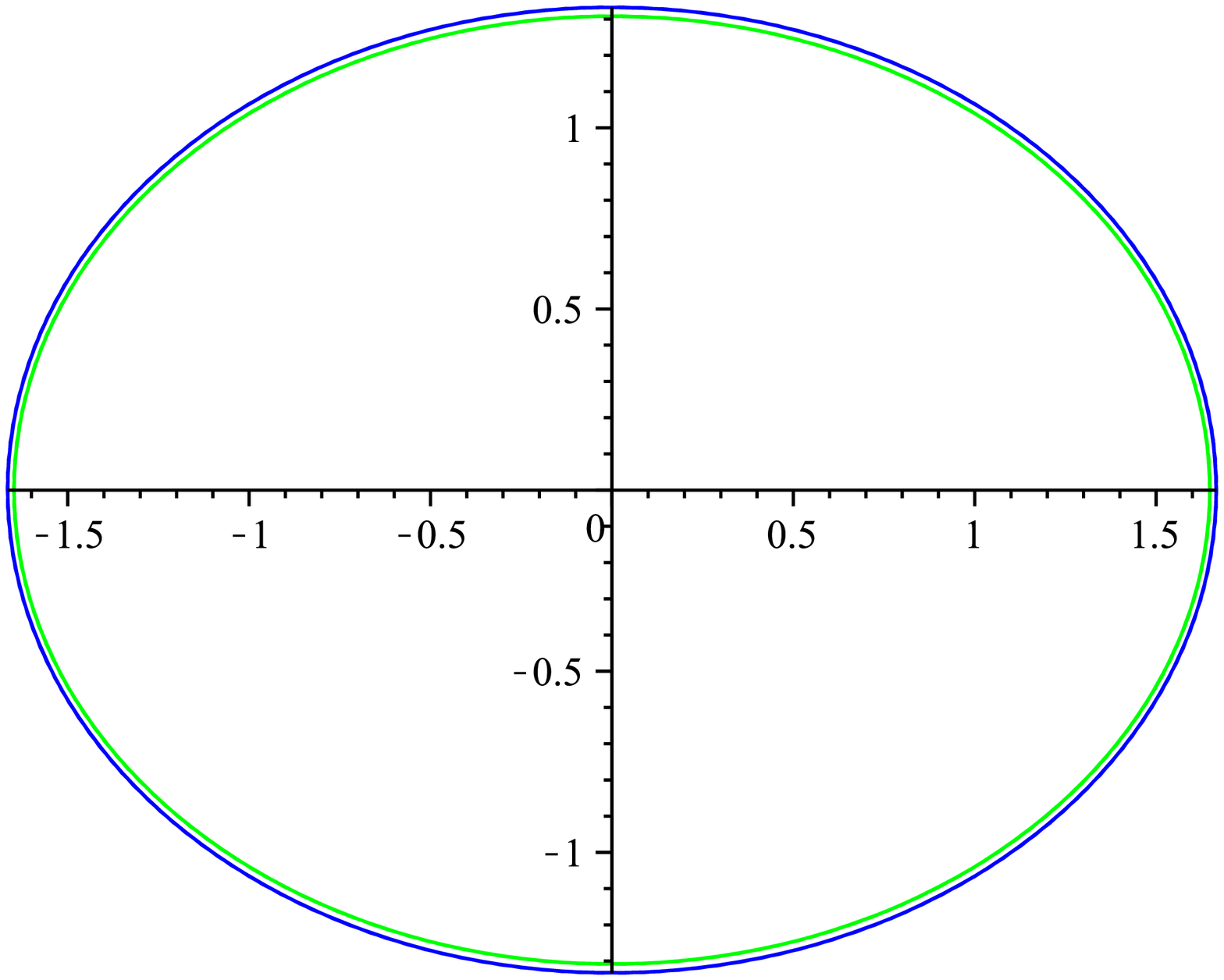}}
\end{minipage}
\qquad\qquad\qquad\qquad\qquad
\begin{minipage}{2.5cm}
{\includegraphics[width=2.20\linewidth]{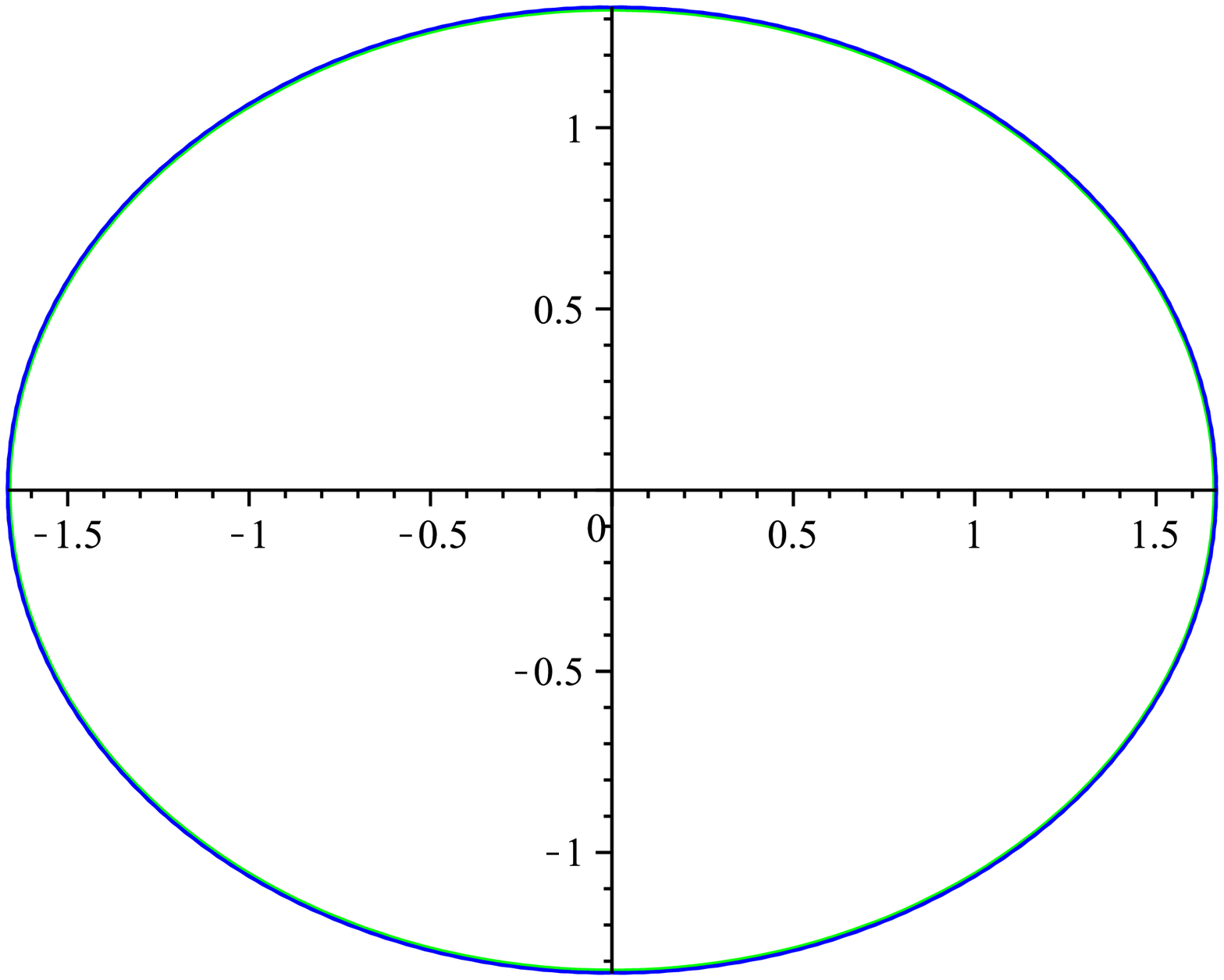}}
\end{minipage}\qquad\qquad\qquad\qquad\qquad
\caption{Recovery of an  ellipse, with $n=10$ (left) and $n=20$ (right).}
\label{fig:ellipse}
\end{center}
\end{figure}

In Figure~\ref{fig:ellipse} we illustrate the reconstruction of an ellipse, where, with the notation of Theorem \ref{thm:recalg}, $\varrho=1/3$.

\begin{figure}
\begin{center}
\includegraphics[width=0.60\linewidth]{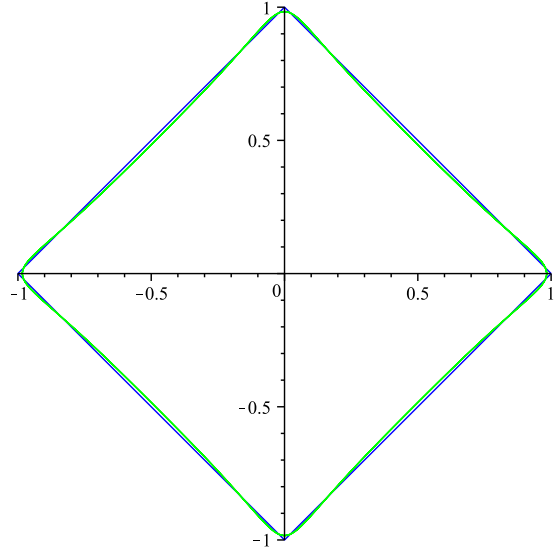}
\caption{Recovery of a  square, with $n=16$.}
\label{fig:square}
\end{center}
\end{figure}
In Figure~\ref{fig:square} we reconstruct a square by using the complex moments up to the degree $16$. We have chosen
$n=16$, so that the result can be compared with the recovery of a square, as shown on page 1067 of
\cite{GHMP}, obtained using the \textit{Exponential Transform Algorithm}.
This is another
reconstruction algorithm based on moments.

\begin{figure}
\begin{center}
\begin{minipage}[t]{2.0cm}
{\includegraphics[scale=0.30]{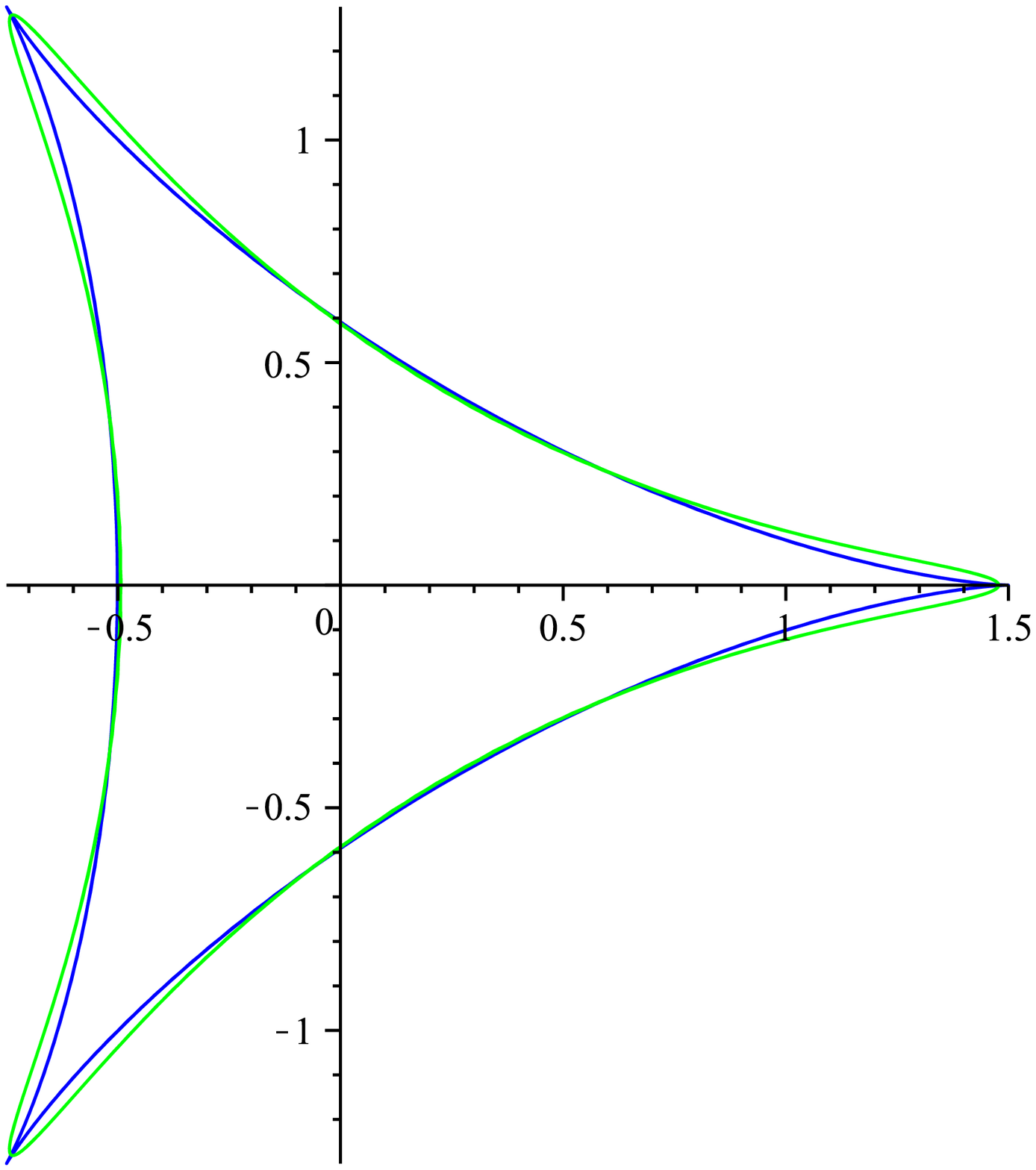}}
\end{minipage}
\qquad\qquad\qquad\qquad\qquad\qquad
\begin{minipage}[t]{2.0cm}
{\includegraphics[scale=0.30]{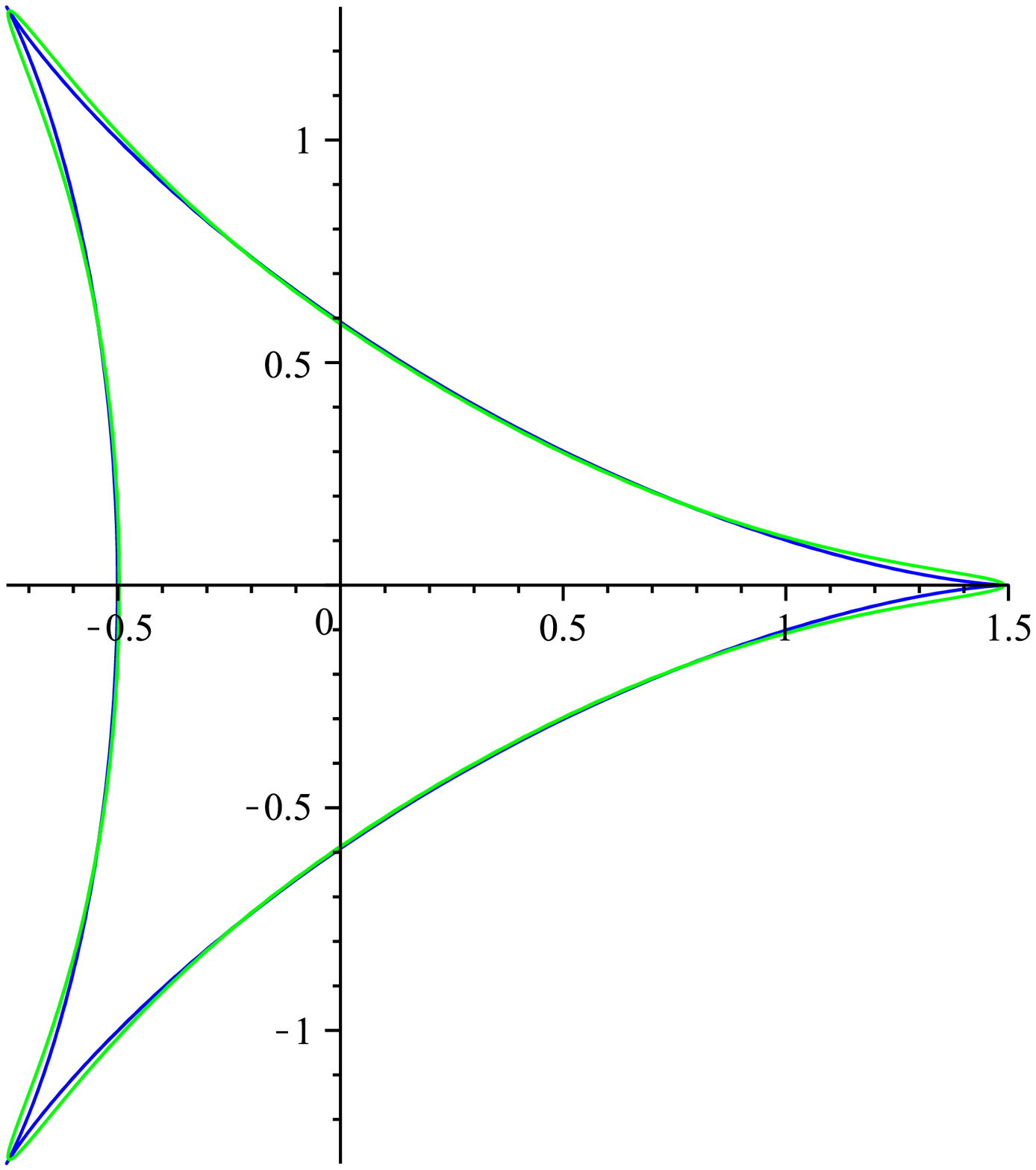}}
\end{minipage}\qquad\qquad\qquad\qquad
\caption{Recovery of a 3-cusped hypocycloid, with $n=20$ (left) and $n=30$ (right).}
\label{fig:cupsed-hypo}
\end{center}
\end{figure}

In order to show that the proposed reconstruction algorithm works equally well for domains where the results of
neither Theorem~\ref{thm:recalg} nor that of Remark~\ref{rem:recalg} apply,
we use it for the recovery of the boundary of the 3-cusped hypocycloid defined by
$$
\Gamma:=\{z=\Psi(w)=w+\frac{1}{2w^2},\quad w\in\mathbb{T}\}.
$$
The application of the algorithm with $n=20$ and $n=30$ is depicted in Figure~\ref{fig:cupsed-hypo}.

Comparing the performance of the above algorithm with that of \cite{St-arXiv1202} for the cases of the ellipse
and the hypocycloid, it appears that the latter algorithm performs slightly better. On the other hand,
both algorithms perform better than the reconstruction algorithms of \cite{GHMP} for the case of the square.
More definitive comparisons will require further experimentation and analysis of all three reconstruction algorithms.

\section{Numerical Results}
In this section we employ the first three steps of the reconstruction algorithm in order to present numerical results
that illustrate the order of convergence in (\ref{th:main-pw-analytic-1}) and (\ref{th:main-pw-analytic-2}),
that is the order in approximating $b$  and $b_2$ by $b^{(n)}$ and $b^{(n)}_2$, respectively.
We consider the case where $\Gamma$ is the equilateral triangle $\Pi_3$ with vertices at $1$, $e^{2i\pi/3}$,
and $e^{4i\pi/3}$. Then, by using the
Schwarz-Christoffel formula it is not difficult to see that the coefficients
$b_n$ of the associated conformal map (\ref{eq:Psi}) are given by $b_0=0$ and
\begin{equation}\label{eq:bn-square}
b_n= \left\{
\begin{array}{ll}
\text{cap}(\Pi_3)(-1)^{m+1}\binom{2/3}{m}\frac{1}{n}, &\text{if } l=1,\\
0, & \text{if } l\ne 1,
\end{array}
\right.
\end{equation}
for $n=3m-l$, $m\in\mathbb{N}$ and $l\in\{0,1,2\}$, where $\binom{2/3}{m}$ denotes the binomial coefficient;
see, e.g., \cite[\S7.8]{St-arXiv1202}. Furthermore, it follows by using the
properties of hypergeometric functions that
\begin{equation}\label{eq:cap-square}
b=\text{cap}(\Pi_3)=\frac{3}{2}\frac{\Gamma(1/3)^3}{4\pi^2}=0.730\,499\,243\,103\cdots,
\end{equation}
where $\Gamma(x)$ denotes the Gamma function with argument $x$.

By using the rotational property of the equilateral triangle, as this is reflected in the relation (\ref{eq:rotsym-pn}),
it is easy to see that
$$
b_{n-k,n}=0, \quad \mbox{ if }\quad k\notin\{2,5,8,\ldots\}.
$$
This is actually the reason why we consider the two approximations $b^{(n)}$ and $b^{(n)}_2$.
Accordingly, we  let $t^{(n)}$ and $t_2^{(n)}$ denote the two errors
\begin{equation}\label{eq:tn-dfn}
t^{(n)}:=b-b^{(n)} \quad\mbox{ and }\quad t_2^{(n)}:=b_2-b_2^{(n)}.
\end{equation}
Then, from Theorem~\ref{th:main-pw-analytic} we have that
\begin{equation}\label{eq:tn-decay}
|t^{(n)}|\le c(\Gamma)\frac{1}{n}\quad \mbox{ and }\quad
|t_2^{(n)}|\le c(\Gamma)\frac{1}{\sqrt{n}},\quad n\in\mathbb{N}.
\end{equation}

In Tables~\ref{tab:cap-square} and \ref{tab:cap-square1} we report the computed values of $b^{(n)}$, $t^{(n)}$
and $b_2^{(n)}$, $t_2^{(n)}$, with $n$ varying from $100$ to $200$. We also report the values of the parameter $s$,
which is designed to test the two hypotheses
$$
|t^{(n)}|\approx 1/n^s\quad\mbox{ and }\quad |t_2^{(n)}|\approx 1/n^s.
$$
This was done by estimating $s$ by means of the two formulae
$$
s_n:=\log\left(\frac{|t^{(n)}|}{|t^{(n+10)}|}\right)\big/\log\left(\frac{n+10}{10}\right)\quad\mbox{and}\quad
s_n:=\log\left(\frac{|t_2^{(n)}|}{|t_2^{(n+10)}|}\right)\big/\log\left(\frac{n+10}{10}\right).
$$

In view of Remark~\ref{rem:arnoldi}, regarding the stability properties of the Arnoldi GS process,
we expect all the figures quoted in the tables to be correct.

\begin{table}
\begin{tabular}{|c|c|c|c|}
\hline $ $
$n$& $b^{(n)}$ & $t^{(n)}$ & $s$  ${\vphantom{\sum^{\sum^{\sum^N}}}}$ \\*[3pt] \hline
100 & 0.730\,487\,539 & 1.17e-05  & 1.9627      \\
110 & 0.730\,489\,536 & 9.70e-06  & 1.9659 \\
120 & 0.730\,491\,062 & 8.18e-06  & 1.9685 \\
130 & 0.730\,492\,255 & 6.98e-06  & 1.9708 \\
140 & 0.730\,493\,204 & 6.03e-06  & 1.9728\\
150 & 0.730\,493\,973 & 5.26e-06  & 1.9745 \\
160 & 0.730\,494\,603 & 4.63e-06  & 1.9761\\
170 & 0.730\,495\,127 & 4.11e-06  & 1.9774 \\
180 & 0.730\,495\,567 & 3.67e-06  & 1.9786\\
190 & 0.730\,495\,940 & 3.30e-06  & 1.9799\\
200 & 0.730\,496\,259 & 2.98e-06  &  --- \\
 \hline
\end{tabular}

\medskip
\caption{Equilateral triangle:  Errors and rates in approximating $b=0.730\,499\,243\,103\cdots$ by $b^{(n)}$.}
\label{tab:cap-square}
\end{table}

\begin{table}
\begin{tabular}{|c|c|c|c|}
\hline $ $
$n$& $b_2^{(n)}$ & $t_2^{(n)}$ & $s$  ${\vphantom{\sum^{\sum^{\sum^N}}}}$ \\*[3pt] \hline
100 & 0.243\,555\,903 & -5.61e-05  & 1.9873      \\
110 & 0.243\,546\,213 & -4.64e-05  & 1.9886      \\
120 & 0.243\,538\,830 & -3.90e-05  & 1.9897      \\
130 & 0.243\,533\,076 & -3.33e-05  & 1.9907      \\
140 & 0.243\,528\,504 & -2.87e-05  & 1.9914      \\
150 & 0.243\,524\,812 & -2.50e-05  & 1.9921      \\
160 & 0.243\,521\,788 & -2.20e-05  & 1.9926   \\
170 & 0.243\,519\,280 & -1.95e-05  & 1.9931  \\
180 & 0.243\,517\,177 & -1.74e-05  & 1.9936   \\
190 & 0.243\,515\,396 & -1.56e-05  & 1.9939  \\
200 & 0.243\,513\,875 & -1.41e-05  & ---   \\
\hline
\end{tabular}

\medskip
\caption{Equilateral triangle: Errors and rates in approximating $b_2:=0.243\,499\,747\,701\cdots$ by $b_2^{(n)}$.}
\label{tab:cap-square1}
\end{table}

It is interesting to note the following regarding the presented results:
\begin{itemize}
\item
The values of $b^{(n)}$ decay monotonically to $b$.
\item
The values of $b_2^{(n)}$ increase monotonically $b_2$.
\item
The values of the parameter $s$ indicate clearly that
$$
|t^{(n)}|\approx 1/n^2\quad\mbox{ and }\quad |t_2^{(n)}|\approx 1/n^2,
$$
This suggests that the two estimates
$$
|t^{(n)}|\le c(\Gamma)\frac{1}{n}\quad \mbox{ and }\quad
|t_2^{(n)}|\le c(\Gamma)\frac{1}{\sqrt{n}},\quad n\in\mathbb{N},
$$
predicted by Theorem~\ref{th:main-pw-analytic} are pessimistic.
\end{itemize}

\section{Proofs}\label{sec:proofs}
\subsection{Proof of Theorem~\ref{th:main-pw-analytic}}
The derivation in the case where $\Gamma$ is piecewise analytic without cusps is based
on results from \cite{St-CR10} and \cite{St-arXiv1202}.
In particular, we utilize the following fact (see \cite[Thm 1.1]{St-CR10}):
\begin{equation}\label{eq:finelambdan}
{\frac{n+1}{\pi}\frac{\gamma^{2(n+1)}}{\lambda_n^2}=1- \alpha_n,}
\end{equation}
where
\begin{equation}\label{eqinthm:finelambdanii}
0\le\alpha_n\le c_1(\Gamma)\,\frac{1}{n}.
\end{equation}
We also note the following estimate \cite[Thm 1.2]{St-CR10}:
\begin{equation}\label{eqinthm:finepn}
{p_n(z)=\sqrt{\frac{n+1}{\pi}}\,\Phi^n(z)\Phi^\prime(z)
\left\{1+A_n(z)\right\}},\quad z\in\Omega,
\end{equation}
where
\begin{eqnarray} \label{eqinthm:finepnii1}
|A_n(z)|\le \frac{c_2(\Gamma)}{\dist(z,\Gamma)\,|\Phi^\prime(z)|}\,\frac{1}{\sqrt{n}}
+c_3(\Gamma)\,\frac{1}{n} .
\end{eqnarray}
Recall that we use $c_k(\Gamma)$, $k=1,2,\ldots$, to denote a non-negative constant that depends on $\Gamma$ only.


The result (\ref{th:main-pw-analytic-1}) follows immediately from (\ref{eq:bn+1n}) and (\ref{eq:finelambdan})--(\ref{eqinthm:finelambdanii}).
For the general case (\ref{th:main-pw-analytic-2}), our proof relies on the use of the auxiliary polynomial
\begin{equation}\label{eq:qndef}
q_{n-1}(z):=G_n(z)-\frac{\gamma^{n+1}}{\lambda_n}p_n(z),\quad n\in\mathbb{N}.
\end{equation}
This is a polynomial of degree at most $n-1$, but it can be identically zero, as the special case when $G$ is a disk
shows.

Next we fix $k=0,1,2,\ldots$. Then from (\ref{eq:Gn-rr}), in conjunction with (\ref{eq:qndef}) and the orthogonality of $p_n$,
we deduce for any $n=k,k+1,k+2,\ldots$, that
\begin{align}\label{eq:expbnk}
\frac{\gamma^{n+1}}{\lambda_n}b_{n-k,n}
&=\langle z\frac{\gamma^{n+1}}{\lambda_n}p_n,p_{n-k}\rangle=\langle zG_n-zq_{n-1},p_{n-k}\rangle\nonumber\\
&=\langle zG_n,p_{n-k}\rangle-\langle zq_{n-1},p_{n-k}\rangle\nonumber\\
&=b\langle G_{n+1},p_{n-k}\rangle+
   \sum_{j=0}^k b_j\langle G_{n-j},p_{n-k}\rangle-\langle zq_{n-1},p_{n-k}\rangle.
\end{align}

Thus, it remains to estimate the two different types of inner products appearing in (\ref{eq:expbnk}), namely
$\langle p_l, G_m\rangle$ and
$\langle zq_m, p_l\rangle$.
This is the objective of the following two lemmas.

\begin{lemma}\label{lem:pkGm}
Assume that $\Gamma$ is piecewise analytic without cusps. Then,
for $l=0,1,2,\ldots$, it holds that
\begin{equation}\label{lem:pkGm-an}
\langle p_l, G_m\rangle= \left\{
\begin{array}{ll}
{\gamma^{m+1}}/{\lambda_m}, & m=l,\\
\xi_m, & m=l+1,l+2,\ldots,
\end{array}
\right.
\end{equation}
where
\begin{equation}\label{eq:rho-esti}
|\xi_m|\le c_1(\Gamma)\frac{1}{m}.
\end{equation}
\end{lemma}

\begin{proof}
For the special case where $m=l$ the result is a trivial consequence of the orthonormality property of the
polynomial $p_m$ and the fact that $G_m$ is a polynomial of exact degree $m$ with leading coefficient $\gamma^{m+1}$.
That is,

\begin{align}\label{eq:pmGm}
\langle p_m, G_m\rangle
&=\langle p_m, \gamma^{m+1}z^m+\cdots\rangle=\langle p_m, \gamma^{m+1}z^m\rangle\nonumber\\
&=\gamma^{m+1}\langle p_m, \frac{1}{\lambda_m}p_m\rangle=\frac{\gamma^{m+1}}{\lambda_m}.
\end{align}
Assume now that $m\in\{l+1,l+2,\ldots\}$. Then, an application of Green's formula, the splitting (\ref{eq:PhinPhipGn})
and the residue theorem give:
\begin{align}\label{eq:Gm-Em}
\langle p_l, G_m\rangle
&=\int_G p_l(z)\overline{G_m(z)}dA(z)=\int_G p_l(z)\overline{\frac{F^\prime_{m+1}(z)}{m+1}}dA(z)\nonumber\\
&=\frac{\pi}{m+1}\left\{\frac{1}{2\pi i}\int_\Gamma p_l(z) \overline{F_{m+1}(z)}dz \right\}\nonumber\\
&=\frac{\pi}{m+1}\left\{\frac{1}{2\pi i}\int_\Gamma p_l(z) \overline{\Phi^{m+1}(z)}dz
+\frac{1}{2\pi i}\int_\Gamma p_l(z) \overline{E_{m+1}(z)}dz \right\}\nonumber\\
&=\frac{\pi}{m+1}\left\{\frac{1}{2\pi i}\int_\Gamma \frac{p_l(z)}{\Phi^{m+1}(z)}dz
+\frac{1}{2\pi i}\int_\Gamma p_l(z) \overline{E_{m+1}(z)}dz \right\}\nonumber\\
&=\frac{1}{2(m+1)i}\int_\Gamma {p_l(z)} \overline{E_{m+1}(z)}dz.
\end{align}

To conclude the proof we use the estimate given in \cite[Lem.~2.5]{St-arXiv1202}, to obtain
\begin{equation}\label{eq:Em-Hm}
\left|\frac{1}{2 i}\int_\Gamma p_l(z)\overline{E_{m+1}(z)}dz\right|\le c_2(\Gamma)\|p_l\|_{L^2(G)}
\left[\int_\Omega|E^\prime_{m+1}(z)|^2dA(z)\right]^{1/2},
\end{equation}
where we made use of the fact that $E^\prime_{m+1}\in L^2(\Omega)$ (see Lemma~\ref{lem:Hn}) and  that a piecewise  analytic without cusps Jordan curve is
quasiconformal and rectifiable.

Therefore, from (\ref{eq:Gm-Em}), the second relation in (\ref{eq:GnFn+1}) and (\ref{eq:Em-Hm}), we have
\begin{equation}\label{eq:plGm}
\left|\langle p_l, G_m\rangle\right|\le c_3(\Gamma)\left[\int_\Omega|H_m(z)|^2dA(z)\right]^{1/2},
\end{equation}
and the required result follows, because the last integral is $O(1/m^2)$; see
\cite[Thm 2.4]{St-arXiv1202}.
\end{proof}

\begin{lemma}\label{lem:zqmpl}
Assume that $\Gamma$ is piecewise analytic without cusps. Then,
for every $m\in\mathbb{N}$ and $l=0,1,2,\ldots$, it holds that
\begin{equation}\label{lem:zqmpl-1}
|\langle zq_m, p_l\rangle|\le c_1(\Gamma)\frac{1}{m}.
\end{equation}
\end{lemma}

\begin{proof}
The result is a simple consequence of Corollary~2.1 in \cite{St-arXiv1202} which states
$$
\|q_m\|_{L^2(G)}\le c_2(\Gamma)\frac{1}{m},
$$
and the Cauchy-Schwarz inequality:
$$
|\langle zq_m, p_l\rangle|\le \|zq_m\|_{L^2(G)}\|p_l\|_{L^2(G)}\le \max\{|z|:z\in\Gamma\}\,\|q_m\|_{L^2(G)}.
$$
\end{proof}

Returning to the proof of Theorem~\ref{th:main-pw-analytic}, we apply the results of the two previous lemmas
to (\ref{eq:expbnk}) and use (\ref{eq:bn-NS-Estimate}) to obtain:
\begin{align}\label{eq:bnkn-11}
\frac{\gamma^{n+1}}{\lambda_n}b_{n-k,n}
&=b\langle G_{n+1},p_{n-k}\rangle+\sum_{j=0}^{k-1} b_j\langle G_{n-j},p_{n-k}\rangle
+ b_k\frac{\gamma^{n-k+1}}{\lambda_{n-k}}
-\langle zq_{n-1},p_{n-k}\rangle\nonumber\\
&=O\left(\frac{1}{n}\right)+\sum_{j=1}^{k-1}O\left(\frac{1}{(n-j)\,j^{1+\omega}}\right)
+b_k\frac{\gamma^{n-k+1}}{\lambda_{n-k}},
\end{align}
where $0<\omega<2$, and $O$ does not depend on $n$ or $k$. Furthermore,
from (\ref{eq:finelambdan})--(\ref{eqinthm:finelambdanii}) we have:
\begin{equation}\label{eq:goverl-1}
\frac{\gamma^{n+1}}{\lambda_n}=\sqrt{\frac{\pi}{n+1}}\left[1+O\left(\frac{1}{n}\right)\right]
\end{equation}
and
\begin{equation}\label{eq:goverl-2}
\frac{\lambda_{n-k}}{\gamma^{n-k+1}}=\sqrt{\frac{n-k+1}{\pi}}\left[1+O\left(\frac{1}{n-k+1}\right)\right].
\end{equation}
Thus, by multiplying both sides of (\ref{eq:bnkn-11}) by $\lambda_{n-k}/\gamma^{n-k+1}$  we get
\begin{equation*}\label{eq:bnkn-12}
\frac{\lambda_{n-k}}{\gamma^{n-k+1}}\frac{\gamma^{n+1}}{\lambda_n}b_{n-k,n}
=b_k+\frac{\lambda_{n-k}}{\gamma^{n-k+1}}\left[O\left(\frac{1}{n}\right)+
\sum_{j=1}^{k-1}O\left(\frac{1}{(n-j)\,j^{1+\omega}}\right)\right],
\end{equation*}
which, in view of the estimates (\ref{eq:goverl-1})--(\ref{eq:goverl-2}), yields for $n\ge k\ge 0$, $n\ge 1$, that
\begin{align}\label{eq:0const}
\sqrt{\frac{n-k+1}{n+1}}&b_{n-k,n}=b_k\left[1+O\left(\frac{1}{n-k+1}\right)\right]\nonumber\\
+&O(\sqrt{n-k+1})\left[O\left(\frac{1}{n}\right)+\sum_{j=1}^{k-1}
O\left(\frac{1}{(n-j)\,j^{1+\omega}}\right)\right],
\end{align}
where an empty sum equals zero.
This leads, for fixed $k$ and $n\to\infty$, to the required estimate (\ref{th:main-pw-analytic-2}),
where now $O$ depends on $k$ but not on $n$.
\qed

\subsection{Proof of Theorem~\ref{th:main-smooth}}\label{sec:proof-smooth}
If $\Gamma\in C(p+\alpha)$, with $p+\alpha>1/2$,
then the following asymptotic formulas  hold as $n\to\infty$, see \cite[p.\ 19--20]{Su74}:
\begin{equation}\label{eq:finelambdan-sm}
{\sqrt{\frac{n+1}{\pi}}\frac{\gamma^{n+1}}{\lambda_n}=1+O\left(\frac{1}{n^{2(p+\alpha)}}\right)}
\end{equation}
and
\begin{equation}\label{eq:finepn-sm}
p_n(z)=\sqrt{\frac{n+1}{\pi}}\,\Phi^n(z)\Phi^\prime(z)
\left\{1+O\left(\frac{\log n}{n^{p+\alpha}}\right)\right\},\quad z\in\overline{\Omega}.
\end{equation}

The proof of the theorem goes along similar lines as the proof of Theorem~\ref{th:main-pw-analytic} given above.
More precisely, for deriving the result for $b_{n+1,n}$ we use the estimate (\ref{eq:finelambdan-sm}) in the place of
(\ref{eq:finelambdan})--(\ref{eqinthm:finelambdanii}).

For the general case $k=0,1,\ldots$, we need estimates for the inner products
$\langle p_l, G_m\rangle$ and  $\langle zq_m, p_l\rangle$. This is done in the following two lemmas, which play the
role of Lemma~\ref{lem:pkGm} and Lemma~\ref{lem:zqmpl} in the proof of Theorem~\ref{th:main-pw-analytic}.

\begin{lemma}\label{lem:pkGm-sm}
Assume that $\Gamma\in C(p+1,\alpha)$, with $p+\alpha>1/2$, then
for $l=0,1,2,\ldots$, it holds that
\begin{equation}\label{eq:pkGm-sm}
\langle p_l, G_m\rangle= \left\{
\begin{array}{ll}
{\gamma^{m+1}}/{\lambda_m}, & m=l,\\
\xi_m, & m=l+1,l+2,\ldots,
\end{array}
\right.
\end{equation}
where
\begin{equation}\label{eq:rho-esti-sm}
|\xi_m|\le c_1(\Gamma)\frac{1}{m^{p+\alpha+1/2}}.
\end{equation}
\end{lemma}

\begin{proof}
The result for $m=l$ is established in Lemma~\ref{lem:pkGm}. Hence, we only consider the case $m=l+1,l+2,\ldots$.

The following estimate has been obtained by Suetin for $\Gamma\in C(p+\alpha)$;
see  \cite[Lem.~1.5]{Su74}:
\begin{equation}\label{eq:Suetin-est}
\left|\frac{1}{2\pi i}\int_\Gamma H_m(z) \overline{E_{m+1}(z)}dz\right|\le
c_2(\Gamma) \frac{1}{m^{2(p+\alpha)}}.
\end{equation}
By using Green's formula in the unbounded domain $\Omega$, together with  (\ref{eq:GnFn+1}), it is readily seen that
\begin{equation}\label{eq:St-en}
\frac{1}{2\pi i}\int_\Gamma H_m(z) \overline{E_{m+1}(z)}dz=-\frac{m+1}{\pi}\int_\Omega|H_m(z)|^2dA(z).
\end{equation}
Hence, from (\ref{eq:Suetin-est}),
\begin{equation}\label{eq:Hn-smooth}
\int_\Omega|H_m(z)|^2dA(z)\le c_3(\Gamma) \frac{1}{m^{2(p+\alpha)+1}},
\end{equation}
and the result (\ref{eq:rho-esti-sm}) follows from the estimate (\ref{eq:plGm}), which is applicable in this case because
any smooth Jordan curve is also quasiconformal and rectifiable.
\end{proof}

\begin{lemma}\label{lem:zqmpl-sm}
Assume that $\Gamma\in C(p+1,\alpha)$. Then for every $m\in\mathbb{N}$ and $l=0,1,2,\ldots$, it holds that
\begin{equation}\label{lem:zqmpl-sm-1}
|\langle zq_m, p_l\rangle|\le c_1(\Gamma)\frac{1}{m^{p+\alpha+1/2}}.
\end{equation}
\end{lemma}

\begin{proof}
As in the proof of Lemma~\ref{lem:zqmpl} we have
$$
|\langle zq_m, p_l\rangle|\le  \max\{|z|:z\in\Gamma\}\,\|q_m\|_{L^2(G)}.
$$
The result of the lemma then follows from (\ref{eq:Hn-smooth}) and the estimate
$$
\|q_m\|_{L^2(G)}\le c_2(\Gamma)\left[\int_\Omega|H_m(z)|^2dA(z)\right]^{1/2},
$$
established in \cite[Thm 2.1]{St-arXiv1202} for domains bounded by a quasiconformal and rectifiable boundary.
\end{proof}

In order to conclude the proof of the theorem, we need an estimate for the decay of the coefficients $b_n$, when the
boundary $\Gamma$ belongs
to the class $C(p+1,\alpha)$, with $p+\alpha>1/2$. This is done in  \cite[Cor. 1.1]{St-arXiv1202}, where it is
shown that
\begin{equation}\label{eq:bh-decay-smooth}
|b_n|\le c_3(\Gamma)\frac{1}{n^{p+\alpha+1/2}},\quad n\in\mathbb{N}.
\end{equation}
Therefore, by using the results for $\langle p_l, G_m\rangle$ and $\langle zq_m,p_l\rangle$, obtained
in the previous two lemmas, together with (\ref{eq:finelambdan-sm})
and (\ref{eq:expbnk}), we see that
\begin{equation*}
\begin{alignedat}{1}
\frac{\gamma^{n+1}}{\lambda_n}b_{n-k,n}
=O\left(\frac{1}{n^{p+\alpha+1/2}}\right)+
\sum_{j=1}^{k-1} O\left(\frac{1}{(j(n-j))^{p+\alpha+1/2}}\right)
+b_k\frac{\gamma^{n-k+1}}{\lambda_{n-k}},
\end{alignedat}
\end{equation*}
where $O$ does not depend on $n$ or $k$. Furthermore, from  (\ref{eq:finelambdan-sm}) we get
\begin{equation}\label{eq:goverl}
\frac{\gamma^{n+1}}{\lambda_n}=\sqrt{\frac{\pi}{n+1}}\left[1+O\left(\frac{1}{n^{2(p+\alpha)}}\right)\right]\\
\end{equation}
and
\begin{equation}\label{eq:goverl2}
\frac{\lambda_{n-k}}{\gamma^{n-k+1}}=\sqrt{\frac{n-k+1}{\pi}}\left[1+O\left(\frac{1}{(n-k+1)^{2(p+\alpha)}}\right)\right].
\end{equation}
The above yield, for $n\ge k\ge 0$, $n\ge 1$, that
\begin{align}\label{eq:0const-sm}
\sqrt{\frac{n-k+1}{n+1}}&b_{n-k,n}=b_k\left[1+O\left(\frac{1}{(n-k+1)^{2(p+\alpha)}}\right)\right]+O(\sqrt{n-k+1})\nonumber\\
&\times\left[O\left(\frac{1}{n^{p+\alpha+1/2}}\right)+\sum_{j=1}^{k-1}
O\left(\frac{1}{(j(n-j))^{p+\alpha+1/2}}\right)\right],
\end{align}
where a empty sum equals zero.
This leads, for fixed $k$ and $n\to\infty$, to the required estimate (\ref{th:main-sm-2}),
where now $O$ depends on $k$ but not on $n$.
\qed

\subsection{Proof of Theorem~\ref{th:main-analytic}} \label{sec:proof-analytic}
Assume that $\Gamma:=\partial G$ is an analytic Jordan curve.
Then the conformal map $\Phi$ has an analytic and univalent continuation across $\Gamma$ in $G$. Let $\varrho<1$ be
defined by
$$
\varrho:=\inf\{r:\ \Phi \mbox{ is analytic and univalent in }\text{ext}(L_\varrho)\setminus\infty\}.
$$
Then the following asymptotic formulas of Carleman~\cite{Ca23} hold as $n\to\infty$:
\begin{equation}\label{eq:finelambdan-an}
{\sqrt{\frac{n+1}{\pi}}\frac{\gamma^{n+1}}{\lambda_n}=1+O(\varrho^{2n})}
\end{equation}
and
\begin{equation}\label{eq:finepn-an}
p_n(z)=\sqrt{\frac{n+1}{\pi}}\,\Phi^n(z)\Phi^\prime(z)
\left\{1+O(\sqrt{n}\varrho^{n})\right\},\quad z\in\overline{\Omega},
\end{equation}
see \cite[p.\ 12]{Gabook87}. In particular,
\begin{equation}\label{eq:finepn-omega}
p_n(z)=\frac{\lambda_n}{\gamma^{n+1}}\,\Phi^n(z)\Phi^\prime(z)
\left\{1+\omega_n(z)\right\},
\end{equation}
where
\begin{equation}\label{eq:omega_n}
\omega_n(z)=\sum_{\nu=1}^{n}\nu A_\nu w^{\nu-1-n}-
\sum_{\nu =1}^\infty\nu a_\nu w^{-\nu-1-n},\quad w=\Phi(z),
\end{equation}
with
\begin{equation}\label{eq:Gaier2.8}
\sum_{\nu =1}^{n}\nu|A_\nu|^2+\sum_{\nu=1}^\infty
\nu|a_\nu|^2\varrho^{-2\nu}\le\frac{\varrho^{2n+2}}{(n+1)(1-\varrho^{2n})},\quad n\in\mathbb{N};
\end{equation}
see \cite[p.\ 15]{Gabook87}.

We fix two different points $z$ and $z_0$ on $\Gamma$ and define
$$
Q_{j+1}(z):=\int_{z_0}^{z}p_j(\zeta)d\zeta
+\frac{\lambda_j}{(j+1)\gamma^{j+1}}\Phi^{j+1}(z_0).
$$
Then, by using integration by parts and the change of variable $w=\Phi(\zeta)$, we have from (\ref{eq:finepn-omega}) and
(\ref{eq:omega_n}) that, for any $j\in\mathbb{N}$,
\begin{equation}\label{eq:Qj+1}
\begin{alignedat}{1}
Q_{j+1}(z)&=\frac{{\lambda_j}}{(j+1)\gamma^{j+1}}\Phi^{j+1}(z)
 +\frac{{\lambda_j}}{\gamma^{j+1}}\int_{z_0}^{z}\Phi^{j}(\zeta)\Phi^\prime(\zeta)\omega_j(\zeta)d\zeta\\
          &=\frac{{\lambda_j}}{(j+1)\gamma^{j+1}}\Phi^{j+1}(z)
 +\frac{{\lambda_j}}{\gamma^{j+1}}\int_{w_0}^{\Phi(z)}w^{j}\omega_j(\Psi(w))dw\\
          &=\frac{{\lambda_j}}{(j+1)\gamma^{j+1}}\Phi^{j+1}(z)
+\frac{{\lambda_j}}{\gamma^{j+1}}
\left[\sum_{\nu=1}^j A_\nu w^\nu+\sum_{\nu=1}^\infty a_\nu w^{-\nu} \right]_{w_0}^{\Phi(z)},
\end{alignedat}
\end{equation}
where $w_0=\Phi(z_0)$.
We claim that for $|w|=1$ there holds
\begin{equation}\label{eq:claim-sumAn}
\left|\sum_{\nu=1}^j A_\nu w^\nu+\sum_{\nu=1}^\infty a_\nu w^{-\nu}\right|
=O\left(\sqrt{\frac{\log (j+1)}{j+1}}\varrho^j\right).
\end{equation}
Indeed,
\begin{align}
\left|\sum_{\nu=1}^j A_\nu w^\nu+\sum_{\nu=1}^\infty a_\nu w^{-\nu}\right|
  &\le\sum_{\nu=1}^j |A_\nu|+\sum_{\nu=1}^\infty |a_\nu|\nonumber\\
 &\le\sqrt{\sum_{\nu=1}^j\nu |A_\nu|^2}\sqrt{\sum_{\nu=1}^j\frac{1}{\nu}}
  +\sqrt{\sum_{\nu=1}^\infty\nu |a_\nu|^2\varrho^{-2\nu}}
  \sqrt{\sum_{\nu=1}^\infty\frac{\varrho^{2\nu}}{\nu}}\nonumber\\
 &\le c_1(\Gamma)\sqrt{\log (j+1)}\sqrt{\sum_{\nu=1}^j\nu |A_\nu|^2}
 +c_2(\Gamma)\sqrt{\sum_{\nu=1}^\infty\nu |a_\nu|^2\varrho^{-2\nu}}\nonumber\\
 &\le c_3(\Gamma)\sqrt{\frac{\log (j+1)}{j+1}}\varrho^j ,
\end{align}
by (\ref{eq:Gaier2.8}), which establishes the claim.

Hence, using the estimate (\ref{eq:finelambdan-an}) we get
\begin{equation}\label{eq:Qj+1-fin}
Q_{j+1}(z)=\frac{\Phi^{j+1}(z)}{\sqrt{\pi (j+1)}}
\left\{1+O\left(\sqrt{(j+1)\log (j+1)}\right)\varrho^j\right\},
\quad z\in\Gamma.
\end{equation}

Next, by Green's formula we have for fixed $k=0,1,\ldots$ and $n\ge k+1$:
\begin{align}\label{eq:kn-analytic}
2\pi i\sqrt{\frac{n-k+1}{n+1}}b_{n-k,n}&=2\pi i\sqrt{\frac{n-k+1}{n+1}}\langle zp_{n},p_{n-k}\rangle\nonumber\\
&=\frac{2\pi i}{2i}\sqrt{\frac{n-k+1}{n+1}}\int_\Gamma zp_n(z)\overline{Q_{n-k+1}(z)}dz\nonumber\\
&=\int_\Gamma z\Phi^n(z)\Phi^\prime(z)\overline{\Phi^{n-k+1}(z)}dz+h_n\nonumber\\
&=\int_\Gamma\frac{\Phi^n(z)}{\Phi^{n-k+1}(z)}\Phi^\prime(z)zdz+h_n\nonumber\\
&=\int_{|w|=1}\frac{w^n}{w^{n-k+1}}\Psi(w)dw+h_n\nonumber\\
&=2\pi ib_k+h_n,
\end{align}
where
\begin{equation}\label{eq:hn}
h_n=O(\sqrt{n})\varrho^{n}+O\left(\sqrt{(n-k+1)\log(n-k+1)}\right)\varrho^{n-k}\{1+O(\sqrt{n})\varrho^{n}\}.
\end{equation}

Thus, for $k\ge 0$ fixed and $\varrho<1$,
\begin{equation}\label{eq:n-analytic}
\sqrt{\frac{n-k+1}{n+1}}b_{n-k,n}=b_k+O(\sqrt{n\log n}\varrho^n),\quad \text{as }n\to\infty.
\end{equation}

It remains to prove (\ref{eq:main-analytic1}). This follows at once from the strong asymptotics for the leading
coefficient (\ref{eq:finelambdan-an}) and the relation (\ref{eq:bn+1n}).
\qed

\subsection{Proof of Theorem~\ref{th:inverse-analytic}} 
We first note that our assumption (\ref{eq-inverse-analytic}), combined with (\ref{eq:bn+1n}), implies that
\begin{equation}\label{eq-inverse-analytic-imp}
\limsup_{n\to\infty}\left|\sqrt{\frac{n+2}{n+1}}\frac{\lambda_{n}}{\lambda_{n+1}}-b\right|^{1/n}<1.
\end{equation}
Now set
$$
\xi_n:=\sqrt{\frac{n+2}{n+1}}\frac{\lambda_{n}}{\lambda_{n+1}}\frac{1}{b}-1,
$$
so that
\begin{equation}\label{eq:xi-dec}
\limsup_{n\to\infty}|\xi_n|^{1/n}<1.
\end{equation}

At the other hand, we have from (\ref{eq:finelambdan}) and (\ref{eq:b-def}) that
\begin{equation*}
(1+\xi_n)^2=\frac{1-\alpha_{n+1}}{1-\alpha_n}.
\end{equation*}
Hence,
\begin{equation*}\label{eq:xi-eq}
\xi_n=\frac{\alpha_n-\alpha_{n+1}}{(1-\alpha_n)(2+\xi_n)},
\end{equation*}
and by using the fact that $\xi_n\to 0$, as $n\to\infty$ together with $0 \le \alpha_n<1$ and
$\alpha_n\to 0$, as $n\to\infty$, we obtain the double inequality
\begin{equation}\label{eq:xi-alpha}
c_1| \alpha_n-\alpha_{n+1}|\le |\xi_n| \le c_2| \alpha_n-\alpha_{n+1}|,
\end{equation}
for some positive constants $c_1$ and $c_2$.

Now, by expanding $\alpha_n$ in the telescoping series
$$
\alpha_n=(\alpha_n-\alpha_{n+1})+(\alpha_{n+1}-\alpha_{n+2})+\cdots,
$$
we conclude, in view of (\ref{eq:xi-dec})--(\ref{eq:xi-alpha}), that
\begin{eqnarray}\label{eq:alpha-week}
\limsup_{n\to\infty} \alpha_n^{1/n} < 1,
\end{eqnarray}
and this, in view of  Theorem~1.3 in \cite{St-arXiv1202} leads to
$$
\limsup_{n\to\infty} |b_n|^{1/n} < 1.
$$

The last inequality implies that the conformal map $\Psi(w)$ has an analytic continuation across $\mathbb{T}$ into $\mathbb{D}$ (see (\ref{eq:Psi})) and thus $\Gamma$ is the analytic image of $\mathbb{T}$.
Therefore, around any $w_0\in\mathbb{T}$, the map $\Psi$ can be represented by a Taylor series expansion of the form
$$
\Psi(w)=\Psi(w_0)+a_1(w-w_0)+a_2(w-w_0)^2+a_3(z-z_0)^3\cdots.
$$
If we had $\Psi^\prime(w_0)=0$, then
$$
\Psi(w)=\Psi(w_0)+a_2(w-w_0)^2+\cdots,
$$
with $a_2\ne 0$, because $\Psi$ is univalent in $\Delta$. These show that $w_0$ would be mapped by $\Psi$ onto an
exterior pointing cusp on $\Gamma$. Since, by assumption, this cannot happen, we see that $\Psi^\prime(w)\neq 0$,
$w\in\mathbb{T}$, which yields the required property that $\Gamma$ is an analytic Jordan curve.
\qed

\subsection{Proof of Theorem~\ref{thm:recalg}}
Recall that $n:=2m$. On $|w|=R$, where $\varrho<1\le R<\infty$, we have from (\ref{eq:Psi}) and (\ref{eq:psi-approx})
\begin{align*}\label{eq:err-Psinm}
|\Psi(w)-\Psi^{(n)}_m(w)|\le |b^{(n)}-b|R+\sum_{k=0}^m\frac{|b^{(n)}_k-b_k|}{R^k}+\sum_{k=m+1}^\infty\frac{|b_k|}{R^k}.
\end{align*}
Therefore, by using the result of Theorem~\ref{th:main-analytic} (see also (\ref{eq:kn-analytic})) and the estimate
$$
|b_k|\le c_1(\Gamma)\frac{\varrho^k}{\sqrt{k}},\quad k\in\mathbb{N};
$$
see \cite[Cor.~1.1]{St-arXiv1202} we get
\begin{equation}\label{eq:final}
|\Psi(w)-\Psi^{(n)}_m(w)|\le c_2(\Gamma)\varrho^{4m}R+c_3(\Gamma)\sqrt{m\log m}\varrho^m+
c_4(\Gamma)\left(\frac{\varrho}{R}\right)^m,
\end{equation}
which yields the desired estimate.
\qed

\bibliographystyle{amsplain}

\def\cprime{$'$}
\providecommand{\bysame}{\leavevmode\hbox to3em{\hrulefill}\thinspace}
\providecommand{\MR}{\relax\ifhmode\unskip\space\fi MR }
\providecommand{\MRhref}[2]{%
  \href{http://www.ams.org/mathscinet-getitem?mr=#1}{#2}
}
\providecommand{\href}[2]{#2}

\end{document}